\documentclass[10pt]{amsart}
\usepackage{geometry}
\usepackage[english]{babel}
\usepackage{graphicx}
\usepackage{amsmath}
\usepackage{amsfonts}

\begin{document}

\newcommand{\wk}{\mbox{$\,<$\hspace{-5pt}\footnotesize )$\,$}}

\numberwithin{equation}{section}
\newtheorem{teo}{Theorem}
\newtheorem{lemma}{Lemma}

\newtheorem{coro}{Corollary}
\newtheorem{prop}{Proposition}
\theoremstyle{definition}
\newtheorem{definition}{Definition}

\theoremstyle{remark}

\newtheorem{remark}{Remark}
\newtheorem{scho}{Scholium}
\newtheorem{open}{Question}

\numberwithin{lemma}{subsection}
\numberwithin{prop}{subsection}
\numberwithin{teo}{subsection}
\numberwithin{definition}{subsection}
\numberwithin{coro}{subsection}
\numberwithin{figure}{subsection}
\numberwithin{remark}{subsection}
\numberwithin{scho}{subsection}

\bibliographystyle{spmpsci}

\title{On a cosine function defined for smooth normed spaces}

\author[V. Balestro]{Vitor Balestro}
\address[V. Balestro]{CEFET/RJ Campus Nova Friburgo
\newline
\indent 28635-000 Nova Friburgo
\newline
\indent Brazil
\newline
\indent\&
\newline
\indent Instituto de Matem\'{a}tica e Estat\'{i}stica
\newline
\indent Universidade Federal Fluminense
\newline
\indent24020-140 Niter\'{o}i
 \newline
\indent Brazil}
\email{vitorbalestro@id.uff.br}
\author[E. Shonoda]{Emad Shonoda}
\address[E. Shonoda]{Department of Mathematics \& Computer Science
\newline
\indent Faculty of Science
\newline
\indent Port Said University
\newline
\indent 42521 Port Said
\newline
\indent Egypt}
\email{en$\_$shonoda@yahoo.de}

\begin{abstract} We continue research on a certain cosine function defined for smooth Minkowski spaces. We prove that such function is symmetric if and only if the corresponding space is Euclidean, and also that it can be given in terms of the Gateaux derivative of the norm. As an application we study the ratio between the lengths of tangent segments drawn from an external point to the unit circle of a Radon plane. We also give a characterization of such planes in terms of signs of the cosine function.
\end{abstract}


\maketitle

\section{Introduction}

Trigonometric functions in Minkowski geometry were already investigated by Finsler \cite{Fi2}, Barthel \cite{Bar2}, Busemann \cite{Bus3}, Guggenheimer \cite{Gug1,Gug2} and Petty \cite{Pet}. They were used to study concepts of curvatures (and curve theory, in general), and they also appeared when studying Minkowski distances given by the solutions of certain second order ordinary differential equation (see \cite{P-B}). Recently, a new sine function introduced by Szostok \cite{szostok} was revisited and used for characterizations of Radon and Euclidean planes, and also for defining new geometric constants (see \cite{Ba-Ma-Te3,bmt,Ba-Ma-Te2}).

Our aim in this paper is to continue recent investigations on the cosine function $\mathrm{cm}$ defined for smooth Minkowski spaces in \cite[Chapter 8]{Tho} (and previously studied by Busemann). We begin by giving a new (equivalent) definition, which is more geometric in nature. Then we relate it to the sine function investigated in \cite{bmt}, in contrast to Busemann's and Thompson's work, where it is compared with another sine function.

Constructing a semi-inner product based on the function $\mathrm{cm}$ yields two interesting results: first, we prove that this function is symmetric if and only if the space is Euclidean. Second, we prove in a similar fashion as for Euclidean space the cosine can be given in terms of the Gateaux derivative of the norm. This can also be seen as a non-local characterization of such derivative.

In Section 3 we investigate the ``distortion" between the lengths of tangent segments drawn from an exterior point to the unit circle of a smooth and strictly convex Minkowski plane. For Radon planes a formula for the ratio between such lengths is given in terms of the function $\mathrm{cm}$, and we characterize the Euclidean plane as the unique for which there is no ``distortion" at all (revisiting a result of Wu \cite{Wu2}). A new, purely geometric property of Radon curves (in terms of parallelism) is also given.

Further interesting topics are dealt with in Section 4. We discuss some Euclidean properties of $\mathrm{cm}$ and investigate a symmetric cosine function based on $\mathrm{cm}$. Inspired by this, a simple characterization of Radon planes in terms of $\mathrm{cm}$ is also given. Finally, we study differentiation of trigonometric functions. An interesting interpretation of Radon planes as the ones for which the ``difference" to the Euclidean plane appears in second order derivatives is also given.

Now we introduce some basic concepts and notation. A \emph{Minkowski (or normed) space} is a finite dimensional real vector space $X$ endowed with a norm $||\cdot||$. Given such a space $X$ we write $X_{o} = X\setminus\{o\}$, where $o$ denotes the origin. We always denote the (closed) unit ball and the unit sphere (circle, if the dimension is $2$) of a normed space by $B$ and $S$, respectively. A segment joining two points $a$ and $b$ is denoted by $\mathrm{seg}[a,b]$, the line passing through the same points is denoted by $\left<a,b\right>$ and we use $\left[a,b\right>$ to represent a ray starting at $a$ and passing through $b$. In a normed space $(X,||\cdot||)$, we say that a vector $x$ is \emph{Birkhoff orthogonal} to a vector $y$ (denoted by $x\dashv_B y$) provided $||x+\lambda y|| \geq ||x||$ for any $\lambda \in \mathbb{R}$. A pair of independent vectors $x,y \in X_o$ for which $x \dashv_B y$ and $y \dashv_B x$ is called a \emph{conjugate pair} (see \cite{martini1} for a proof of the existence of such a base in any normed plane). A \emph{Radon plane} is a Minkowski plane for which Birkhoff orthogonality is symmetric. We assume that a Minkowski plane $(X,||\cdot||)$ is endowed with a non-degenerate symplectic bilinear form $[\cdot,\cdot]:X\times X\rightarrow \mathbb{R}$ (which is unique up to constant multiplication), and define the associated antinorm in $X$ by $||\cdot||_a = \sup\{|[\cdot,y]|:y\in S\}$. If the plane is Radon, we always assume that $[\cdot,\cdot]$ is rescaled in such a way that we have $||\cdot||_a=||\cdot||$ (see \cite{martiniantinorms}). Given an angle $\wk\mathbf{xoy}$ formed by two rays $\left.[o,x\right>$ and $\left.[o,y\right>$ we define its \emph{Glogovskii angular bisector} by the the locus of the centers of the circles inscribed in $\wk\mathbf{xoy}$ and its \emph{Busemann angular bisector} to be the ray $\left.\left[o,\frac{x}{||x||}+\frac{y}{||y||}\right>\right.$.

For general references to Minkowski geometry we cite the surveys \cite{martini2} and \cite{martini1}, as well as the book \cite{Tho}. For orthogonality in such spaces we refer the reader to \cite{alonso}, and to antinorms and Radon planes we recommend \cite{martiniantinorms}.

\section{The function cm}

\subsection{An equivalent definition} \label{define}

Thompson \cite[Chapter 8]{Tho} defines a cosine function $\mathrm{cm}:X_{o}\times X_{o}\rightarrow\mathbb{R}$ in a smooth Minkowski space by 
\begin{align*} \mathrm{cm}(x,y) = \frac{f_x(y)}{||f_x||_*||y||}, \end{align*}
where, up to a positive scalar multiplication, $f_x \in X^*$ is the unique linear functional which attains its norm at $\frac{x}{||x||}$, and $||\cdot||_*$ is the usual norm in the dual space $X^*$. This subsection is devoted to present an equivalent definition of the function $\mathrm{cm}$ which is, from our point of view, a little easier to work with. Also, this new definition establishes an interesting interplay with a variant of the sine function studied in \cite{szostok} and \cite{bmt}, as will become clear within the next subsection. \\

Assume that $(X,||\cdot||)$ is a smooth Minkowski plane, i.e., a plane for which there is only one supporting line to the unit ball $B$ at each point of the unit circle $S$. Due to smoothness, Birkhoff orthogonality is right unique, i.e., for each non-zero $x \in X$ there exists precisely one direction $y \in X$ such that $x \dashv_B y$. Thus, for such a space we may define a (continuous) map $b:X_o\rightarrow X_o$ which associates each $x \in X_o$ to the unique vector $b(x)\in X_o$ such that $||b(x)||_a = 1$, $x \dashv_B b(x)$ and $[x,b(x)] > 0$. Given $x \in X_o$ it is clear that the functional $y \mapsto [y,b(x)]$ assumes its norm at $\frac{x}{||x||}$ (cf. \cite{martiniantinorms}), and hence we can set
\begin{align}\label{cm1} \mathrm{cm}(x,y) = \frac{[y,b(x)]}{||y||}. \end{align}

We still can make an equivalent definition which is more geometric. We just have to use an argument as in \cite[Proposition 3.1]{bmt}: we write
\begin{align*} ||f_x|| = \sup_{t\in\mathbb{R}}\frac{[y+tb(x),b(x)]}{||y+tb(x)||} = \frac{[y,b(x)]}{\inf_{t\in\mathbb{R}}||y+tb(x)||}, \end{align*}
to obtain finally
\begin{align}\label{cm2} \mathrm{cm}(x,y) = \mathrm{sgn}\left([y,b(x)]\right)\frac{\inf_{t\in\mathbb{R}}||y+tb(x)||}{||y||}, \end{align}
where $\mathrm{sgn}$ denotes the usual sign function. The geometric interpretation is the following: if $x,y \in S$ are unit vectors, then the value of $\mathrm{cm}(x,y)$ is the distance from the parallel, passing through $y$, of the line which supports the unit circle at $x$ to the origin (see Figure \ref{fig1cosine}). \\

Now, we naturally extend this definition to spaces of higher dimension as follows: if $x,y \in X_o$, then $\mathrm{cm}(x,y)$ is calculated in a plane spanned by $x$ and $y$ endowed with the induced norm. Of course, such a plane is not unique if and only if $x$ and $y$ are dependent, and in this case the value of $\mathrm{cm}(x,y)$ is independent of the considered plane. Having said all this, it is clear that throughout the text every calculation will be made in planes, also when we are dealing with spaces with dimension $\geq 3$.

\begin{figure}

\includegraphics{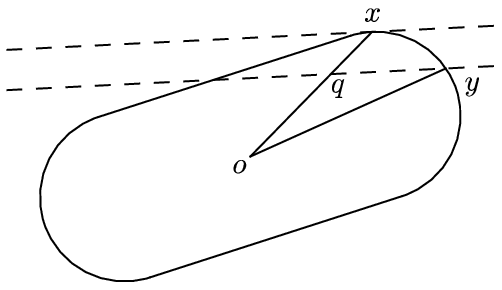}
\caption{$||q||=\mathrm{cm}(x,y)$}
\label{fig1cosine}
\end{figure}

\begin{remark} We could have also (equivalently) defined the function $\mathrm{cm}$ ``externally" as follows: consider $x,y \in S$. Then, we have that $\mathrm{cm}(x,y)$ is, up to sign, the inverse of the length of the segment whose endpoints are the origin and the intersection of the ray $\left.[o,y\right>$ and the supporting line to $S$ at $x$; or we have $\mathrm{cm}(x,y)=0$ if such an intersection does not exist. Petty and Barry (see (2.4) in \cite{P-B}) studied unit circles given by solutions of second-order differential equations of the type $u'' + R(t)u = 0$, giving suitably generalized trigonometric functions. Considering this ``external" definition, it is shown there that the given cosine function agrees with the function $\mathrm{cm}$, as it was noticed by Thompson \cite[Section 8.5]{Tho}.
\end{remark}

 We finish this subsection presenting (already known) early properties of the function $\mathrm{cm}$ concerning Birkhoff orthogonality and strict convexity of the norm. 

\begin{lemma}\label{lemmacosine}\normalfont We have $-1\leq \mathrm{cm}(x,y) \leq 1$ for every $x,y \in X_o$. We have equality $|\mathrm{cm}(x,y)| = 1$ if and only if $x$ and $y$ are dependent or $\mathrm{seg}\left[\frac{x}{||x||},\frac{y}{||y||}\right] \subseteq S$. Also, $\mathrm{cm}(x,y) = 0$ if and only if $x \dashv_B y$.
\end{lemma}
\begin{proof} The inequality comes straightforwardly. If $|\mathrm{cm}(x,y)| = 1$, then we have $|[y,b(x)]| = ||y||$ and hence $y\dashv_B b(x)$, since $||b(x)||_a = 1$. The desired follows. The remaining is also immediate.

\end{proof}

\subsection{The signed sine function and its relation to cm}

The sine function $s:X_o\times X_o \rightarrow \mathbb{R}$ given by 
\begin{align}\label{classicsine} s(x,y) = \frac{\inf_{t\in\mathbb{R}}||x+ty||}{||x||} \end{align}
was studied in \cite{szostok} and \cite{bmt}. Within this subsection we study a slight modification in this function and study its relations with $\mathrm{cm}$. First, from (\ref{cm2}) we have that $|\mathrm{cm}(x,y)| = s(y,b(x))$. Also, it is known (see \cite{bmt}) that in a normed plane the sine function (\ref{classicsine}) can be given as \\
\begin{align*} s(x,y) = \frac{|[x,y]|}{||y||_a||x||},\end{align*}
where $||\cdot||_a = \sup\{[\cdot,z]:z\in S\}$ is, as usual, the \emph{antinorm} associated to the norm $||\cdot||$. Thus, we may think about a signed sine function, i.e., a function $\mathrm{sn}:X_o\times X_o \rightarrow \mathbb{R}$ defined as 
\begin{align*} \mathrm{sn}(x,y) = \frac{[x,y]}{||y||_a||x||}.\end{align*}
And we readily see that $\mathrm{cm}(x,y) = \mathrm{sn}(y,b(x))$. Notice that these functions coincide with the standard ones in the Euclidean plane (considering that $[\cdot,\cdot]$ is the usual determinant, of course).\\

We notice that we might expect some kind of polar coordinates for the unit circle involving the signed sine and the cosine functions. But this happens only in Radon planes.

\begin{lemma}\label{polarlemma} Assume that $(X,||\cdot||)$ is Radon and, rescaling if necessary, that $[x,b(x)] = 1$ whenever $x \in S$. If $x,y \in S$, we have 
\begin{align}\label{polarcoordinates} y = \mathrm{cm}(x,y)x + \mathrm{sn}(x,y)b(x). \end{align}
This does not necessarily hold if $(X,||\cdot||)$ is not Radon.
\end{lemma}
\begin{proof} Write $y = \alpha x + \beta b(x)$ for $\alpha,\beta \in \mathbb{R}$. We have $\mathrm{cm}(x,y) = \mathrm{sn}(y,b(x)) = \alpha$, since $y \in S$ and $||b(x)||_a = 1$. On the other hand, $\mathrm{sn}(x,y) = \frac{\beta}{||y||_a} = \beta$, since $||\cdot||$ is Radon. 

\end{proof}

If the plane is not Radon, then it is not possible to assume that $[x,y] = 1$ whenever $x \dashv_B y$. Even if we fix a basis $\{x,b(x)\}$ and rescale $[\cdot,\cdot]$ in order to have $[x,b(x)] = 1$, equality (\ref{polarcoordinates}) is not necessarily true. Indeed, we just have to pick some $y \in S$ for which $||y||_a \neq 1$. Although, we have: \\

\begin {coro} Assume that $(X,||\cdot||)$ is a smooth normed plane, and let $ \{z,b(z)\} $ be a conjugate base. If $y \in S$ we have 
\begin{align}\label{conjugatepolarcoordinates} y = \mathrm{cm}(z,y)z + \mathrm{cm}(b(z),y)b(z). \end{align}

\end {coro}

\begin {proof} The proof is straightforward by (\ref{polarcoordinates}).
\end {proof}

In the conditions of the Lemma \ref{polarlemma}, notice that if we fix any $x \in S$ then, in view of (\ref{polarcoordinates}), the lines $\left<-x,x\right>$ and $\left<-b(x),b(x)\right>$ play the role of axes of the trigonometric circle of $X$. When we take $y$ ranging through the unit circle, the functions $y \mapsto \mathrm{sn}(x,y)$ and $y \mapsto \mathrm{cm}(x,y)$ behave similarly to the standard ones in the Euclidean plane.

\subsection{Semi-inner products and the cosine function}

Semi-inner products are natural generalizations of the inner product for normed spaces. A semi-inner product (in the sense of Lumer-Giles, see \cite{dragomirsemiinner}) is defined to be an application $(\cdot,\cdot)_s:X\times X \rightarrow \mathbb{R}$ for which, for any $x,y,z \in X$ and $\alpha,\beta \in \mathbb{R}$, the following holds:\\

\noindent\textbf{(a)} $(z,\alpha x + \beta y)_s = \alpha(z,x)_s + \beta(z,y)_s$; \\

\noindent\textbf{(b)} $(\alpha x, y)_s = \alpha (x,y)_s$;\\

\noindent\textbf{(c)} $(x,x)_s \geq 0$, with equality if and only if $x = 0$; and\\

\noindent\textbf{(d)} $(x,y)_s^2\leq (x,x)_s(y,y)_s$.\\

It is clear that a semi-inner product yields a norm by setting $||\cdot||_s = \sqrt{(\cdot,\cdot)_s}$, and it is known that for every normed space $(X,||\cdot||)$ there exists a semi-inner product $(\cdot,\cdot)_s$ whose associated norm $||\cdot||_s$ equals the original norm $||\cdot||$. It is also known that if the space is smooth, then such a semi-inner product is unique (see \cite{dragomirsemiinner}). In a smooth normed space this norm generating semi-inner product can be given in terms of the function $\mathrm{cm}$ (see \cite{shonoda} and \cite{shonoda-weiss}).

\begin{lemma} Let $(X,||\cdot||)$ be a smooth Minkowski space. The application $(\cdot,\cdot)_s:X\times X \rightarrow \mathbb{R}$ defined by $(x,y)_s = ||x||\cdot||y||\mathrm{cm}(x,y)$ if $x,y \in X_o$ and $(x,y)_s = 0$ if $||x||\cdot||y|| = 0$ is the unique norm generating semi-inner product in X.
\end{lemma}
\begin{proof} Since we have $\mathrm{cm}(x,y) = \mathrm{sn}(y,b(x)) = \frac{[y,b(x)]}{||y||}$ it follows that $(x,y)_s = ||x||\cdot[y,b(x)]$. All the desired properties come easily from this formula. Notice that $b(-y) = -b(y)$ for any $y \in X_o$. The original norm is re-obtained from $(\cdot,\cdot)_s$ since $\mathrm{cm}(x,x) = 1$ for any $x \in X_o$, and the uniqueness property comes immediately from the smoothness hypothesis.

\end{proof}

As a consequence we have a characterization of the spaces where $\mathrm{cm}$ is symmetric. Unlike the sine function, this property does not characterize Radon planes.

\begin{prop} \label{symmetricprop}If the cosine function of a normed space $(X,||\cdot||)$ is symmetric, then the norm is Euclidean.
\end{prop}
\begin{proof} If $\mathrm{cm}$ is symmetric, then the norm is derived from a semi-inner product which is symmetric, and thus it is derived from an inner product. 

\end{proof}

\subsection{The Gateaux derivative of the norm}

In a smooth normed space $(X,||\cdot||)$ the Gateaux derivative of the norm inspires the functional
\begin{align*} g(x,y) = ||x||\lim_{t\rightarrow 0}\frac{||x+ty||-||x||}{t},
\end{align*}
which coincides with the inner product if the norm is Euclidean. Hence, denoting the inner product by $(\cdot,\cdot)$, the Euclidean norm derived from it by $||\cdot||_E$, and its associated functional by $g_E$, we can obtain the standard Euclidean cosine as follows: 
\begin{align*} \mathrm{cos}(x,y) = \frac{(x,y)}{||x||_E||y||_E} = \frac{g_E(x,y)}{||x||_E||y||_E}.
\end{align*}
 
The functional $g$ was extensively studied by Mili\v{c}i\v{c} for general normed spaces (see, e.g., \cite{milicic1973produit,milicic1987gortogonalite,milicic1987generalisation,milicic1990fonctionelle}). The main result of this subsection is the, perhaps surprising, fact that the cosine function of any smooth Minkowski space can be given in terms of the Gateaux derivative of its norm.

\begin{prop}\label{gateauxcm} In any smooth normed space $(X,||\cdot||)$ we have
\begin{align*} \mathrm{cm}(x,y) = \frac{g(x,y)}{||x||.||y||},
\end{align*}
for all $x,y \in X_o$.
\end{prop}
\begin{proof} It is known that the functional $g$ is a semi-inner product which generates the norm. By the uniqueness of such a semi-inner product it follows that $||x||.||y||\mathrm{cm}(x,y) = g(x,y)$. This concludes the proof.

\end{proof}

We highlight here that the main interest in the result probably relies in the fact that, in view of Subsection \ref{define}, we have now an easy geometric interpretation for the Gateaux derivative of the norm. First, if $x,y \in S$ are unit vectors of a normed plane, then the derivative of the norm at $x$ in the direction $y$ is the (signed) distance from the intersection of the parallel to the supporting line to $B$ at $x$ drawn through $y$ with the segment $\mathrm{seg}[-x,x]$ to the origin. Now we extend this approach to spaces of higher dimension in the usual way: for the derivative $g(x,y)$ we repeat the argument in a plane spanned by $x$ and $y$ endowed with the usual norm. \\

To finish this subsection we outline that, in a normed plane, the derivative of the norm can also be characterized by means of the application $b$.

\begin{coro} We have
\begin{align*} [y,b(x)] = \lim_{t\rightarrow0}\frac{||x+ty||-||x||}{t},
\end{align*}
for any $x,y \in X_o$.
\end{coro}

\begin{remark} Since the norm is a radial function, one may expect that its gradient flow is given by the lines which pass through the origin. This is indeed true in the following sense: one can define the gradient of the norm in a point $x \in X_o$ to be the direction $\nabla_{||\cdot||}(x) \in S$ for which the derivative of the norm at $x$ attains its maximum. From the above corollary this vector maximizes the map $y\mapsto[y,b(x)]$ in the unit circle, and hence it follows that $\nabla_{||\cdot||}(x) = \frac{x}{||x||}$.
\end{remark}

\section{The outer distortion functional}

Let $(X,||\cdot||)$ be a smooth and strictly convex normed plane. Then through any point $p \in \mathrm{int}(X\setminus B)$ one can draw exactly two tangent lines to $B$, each of them touching $\partial B$ in precisely one point. Denote these points by $q_1$ and $q_2$. In the Euclidean plane we certainly have $||p-q_1|| = ||p-q_2||$, but it is easy to see that this is not necessarily true for an arbitrary smooth Minkowski plane. Our objective is to study the ratio between pairs of such lengths in an arbitrary (smooth and strictly convex) Minkowski circle.\\

\subsection{Definition and early properties}

Consider the following construction: let $x,y$ be unit vectors such that the rays $\left.[o,x\right>$ and $\left.[o,y\right>$ form an angle (in other words, $x$ and $y$ are linearly independent). Let $C$ be any circle inscribed in this angle intersecting its sides $\left.[o,x\right>$ and $\left.[o,y\right>$ respectively in the points $\beta x$ and $\alpha y$ (see Figure \ref{figcomm2}). Hence, the number $\gamma(x,y) := \frac{||\beta x||}{||\alpha y||} = \frac{\beta}{\alpha}$ is clearly independent of the choice of the inscribed circle $C$.

\begin{figure}[h]

\includegraphics{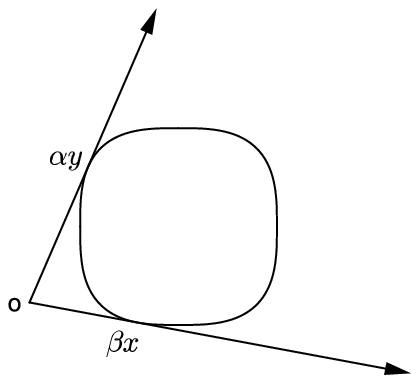}
\caption{$\gamma(x,y) = \frac{||\beta x||}{||\alpha y||}$}
\label{figcomm2}
\end{figure}

Let $D = \{(x,x):x \in S\}$ and $D'=\{(x,-x):x\in S\}$. Then, we define the \textit{outer distortion functional} of the Minkowski plane $(X,||\cdot||)$ to be the function $\gamma:(S\times S)\setminus(D\cup D') \rightarrow \mathbb{R}$ defined as above.

\begin{lemma}[Properties of $\gamma$]\label{propertiesouterfunctional} The outer distortion functional has the following properties:\\

\normalfont

\noindent\textbf{(a)} \textit{$\gamma(x,y) = \gamma(y,x)^{-1}$;} \\

\noindent\textbf{(b)} \textit{$\gamma(x,y) = \gamma(-x,-y)$; and}\\

\noindent\textbf{(c)} \textit{if $x \dashv_B y$ and $\gamma(x,y) = 1$, then $y \dashv_B x$. In particular, a plane is Radon if and only if $x\dashv_B y$ implies $\gamma(x,y) = 1$.}
\end{lemma}

\begin{proof} Assertions \textbf{(a)} and \textbf{(b)} are immediate. For \textbf{(c)}, assume that $x \dashv_B y$ but the converse is not true. Let $P$ be the parallelogram circumscribed to the unit circle whose sides are respectively parallel to $x$ and $y$. The segment connecting the midpoints of the sides in the direction of $y$ is parallel to $x$, and hence its length equals $2$, but the segment which connects the midpoints of the sides in direction of $x$ is not parallel to $y$, and by strict convexity it follows that its length is not $1$. It follows that we cannot have $\gamma(x,y) = 1$ (see Figure \ref{figcomm4}). The remaining part is immediate.

\end{proof}

\begin{figure}

\includegraphics{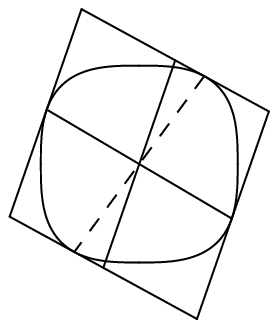}
\caption{$\gamma(x,y) \neq 1$}
\label{figcomm4}
\end{figure}

We finish this subsection showing, by means of a construction, that the outer distortion functional is not uniformly bounded, either from above or below (by a positive constant), for smooth normed planes. Fix numbers $1 < q \leq 2 \leq p < +\infty$ for which $\frac{1}{p} + \frac{1}{q} = 1$, and consider the standard space $\mathbb{R}^2$ endowed with the mixed $l_p-l_q$ norm:
\begin{align*} ||(\alpha,\beta)||_{p,q} = \left\{\begin{array}{ll} \left(|\alpha|^p+|\beta|^p\right)^{1/p} \ \mathrm{if} \ \alpha\beta \geq 0 \\ \left(|\alpha|^q + |\beta|^q\right)^{1/q} \ \mathrm{if} \ \alpha\beta \leq 0 \end{array}\right. .
\end{align*}
 
Let the tangents to the unit circle at the points $a = \left(-\frac{1}{2^{1/q}},\frac{1}{2^{1/q}}\right)$ and $b = (0,1)$ intersect at $c$. Thus, the outer distortion functional $\gamma_{p,q}$ evaluated at the pair of directions $(a-c,b-c)$ is the ratio between the Minkowski lengths of $\mathrm{seg}[a,c]$ and $\mathrm{seg}[b,c]$. It is easy to see that the angular coefficient of the tangent line to the unit circle at $a$ equals 1. Thus $c = \left(1-\frac{2}{2^{1/q}},1\right)$, and a simple calculation shows that
\begin{align*} \gamma_{p,q}(a-c,b-c) = 2^{1/p}\left(1+\frac{2^{-1/q}}{1-2^{1-1/q}}\right),
\end{align*}
where the right handed expression goes to infinity as $p \rightarrow +\infty$. Therefore, the outer distortion functional is not bounded from above (and due to Lemma \ref{propertiesouterfunctional}, also from below by a positive constant) for smooth normed planes. Actually, it is not bounded either from above or below, even in the class of smooth Radon planes (any mixed $l_p-l_q$ plane is Radon, see \cite{martiniantinorms}).

\subsection{The outer distortion of a Radon plane}

In the next theorem we characterize the outer distortion functional of a smooth Radon plane by means of the cosine function.

\begin{teo} \label{distortion} Let $(X,||\cdot||)$ be a smooth and strictly convex Radon plane. Then, for any independent vectors $x,y \in S$, we have
\begin{align}\label{distortionradon} \gamma(x,y) = \frac{\mathrm{cm}(x,x+y)}{\mathrm{cm}(y,x+y)}.
\end{align}
\end{teo}

\begin{proof} Assume first that $X$ is Radon and let $x,y$ be unit independent vectors. Hence, the point $x+y$ is a point of the Glogovskii angular bisector of the angle formed by the rays $\left.[o,x\right>$ and $\left.[o,y\right>$. Thus, there is an inscribed circle $C$ with center $x+y$ touching the sides of the angle in the points $\alpha y$ and $\beta x$, say. Hence, $\alpha y - (x+y) \dashv_B y$ and $\beta x - (x+y) \dashv_B x$. Recalling that Birkhoff orthogonality is symmetric in a Radon plane we have 
\begin{align*} |\mathrm{cm}(x,x+y)|= \frac{\left|[x+y,\beta x - (x+y)]\right|}{||\beta x -(x+y)||_a||x+y||} = \frac{\beta\left|[y,x]\right|}{||\beta x -(x+y)||_a||x+y||}; \ \mathrm{and} \\
\left|\mathrm{cm}(y,x+y)\right| = \frac{\left|[x+y,\alpha y - (x+y)]\right|}{||\alpha y - (x+y)||_a||x+y||} = \frac{\alpha\left|[x,y]\right|}{||\alpha y - (x+y)||_a||x+y||}.
\end{align*}
Thus, since $||\beta x-(x+y)||_a = ||\alpha y - (x+y)||_a$ we have the equality 
\begin{align*} \gamma(x,y) = \left|\frac{\mathrm{cm}(x,x+y)}{\mathrm{cm}(y,x+y)}\right|.
\end{align*}
Now we prove that we actually do not need to consider the absolute value. Indeed, if $x,y \in S$ we have $\mathrm{cm}(x,x+y) = \frac{[x+y,b(x)]}{||x+y||} = \frac{1+[y,b(x)]}{||x+y||}$. Since we are working with a Radon plane we have $|[y,b(x)]| \leq ||b(x)||_a||y|| = 1$, and hence $\mathrm{cm}(x,x+y) \geq 0$. This proves the desired.\\

\end{proof}

\begin{open} Does the converse to the previous theorem hold?
\end{open}

We emphasize in the next corollary that this result yields a geometric property of Radon curves that only relies on elementary concepts.

\begin{coro} Let $R$ be a Radon curve with center $o$ (for the sake of simplicity) and let $t_1, t_2$ be two non-parallel tangent lines intersecting $R$ in $q_1$ and $q_2$, respectively, and meeting at $p$. Through $o$, draw the lines parallel to $t_1$ and $t_2$ and assume that they intersect $R$ in $b_1$ and $b_2$, respectively. Denote by $b$ the intersection of the ray $\left.[o,b_1+b_2\right>$ with $R$. Let $s_1$ and $s_2$ be the supporting lines to $R$ at $b_1$ and $b_2$, respectively, and let the line parallel to $s_1$ through $b$ intersect the segment $\mathrm{seg}[-b_1,b_1]$ at the point $c_1$, and the line parallel to $s_2$ through $b$ intersect the segment $\mathrm{seg}[-b_2,b_2]$ at the point $c_2$. Then the line spanned by the tangency points $q_1$ and $q_2$ is parallel to the line through $c_1$ and $c_2$ (see Figure \ref{fig3cosine}).
\end{coro}

\begin{proof} Assume that $R$ is the unit circle of the Minkowski plane $(X,||\cdot||)$. We have $||c_1|| = |\mathrm{cm}(b_1,b_1+b_2)|$ and $||c_2|| = |\mathrm{cm}(b_2,b_1+b_2)|$. Hence, $\frac{||c_1||}{||c_2||} = \frac{||q_1-p||}{||q_2-p||}$. The desired follows. 

\end{proof}

\begin{remark} In \cite[Proposition 5.2]{Ba-Ma-Te} a characterization of Radon curves in terms of parallelism is given, and here we can give a characterization of such planes in terms of collinearity. Indeed, it is easy to see that the Glogovskii and Busemann angular bisectors always coincide if and only if the points $p$, $o$ and $b$ (as constructed above) are collinear. Since a plane is Radon if and only if these angular bisectors always coincide (see \cite{Duev2}) we have the desired characterization. Notice that this result still holds if the plane is not smooth or strictly convex, since in this case we can still define the tangents from an external point.
\end{remark}

\begin{figure}

\includegraphics{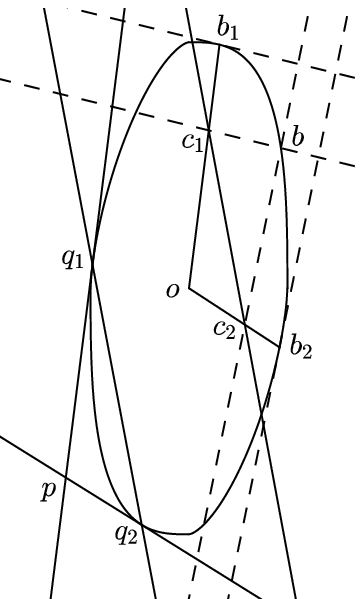}
\caption{Lines $\left<q_1,q_2\right>$ and $\left<c_1,c_2\right>$ are parallel}
\label{fig3cosine}
\end{figure}

Another consequence of Theorem \ref{distortion} is that the outer distortion functional can be continuously extended in a Radon plane. 

\begin{coro} Let $(X,||\cdot||)$ be a Radon plane with outer distortion functional $\gamma:(S\times S)\setminus(D\cup D')\rightarrow\mathbb{R}$. For any $x \in S$ we have
\begin{align*} \lim_{y\rightarrow x} \gamma(x,y) = \lim_{y\rightarrow -x} \gamma(x,y) = 1.
\end{align*}
\end{coro}
\begin{proof}  We write 
\begin{align*} \gamma(x,y) = \frac{1+[y,b(x)]}{1+[x,b(y)]}.
\end{align*}
Hence, the first limit comes straightforwardly. For the second one we apply the L'Hospital Rule (differentiating with respect to the normed arc length parameter, see Subsection \ref{differentiation}) twice to obtain 
\begin{align*} \lim_{y\rightarrow-x}\frac{1+[y,b(x)]}{1+[x,b(y)]} = \lim_{y\rightarrow-x}\frac{[b(y),b(x)]}{[x,-y]} = \lim_{y\rightarrow-x}\frac{[-y,b(x)]}{[x,-b(y)]} = 1.
\end{align*}

\end{proof}

\begin{remark} The above result means, in some sense, that Radon planes are ``locally Euclidean", i.e., for close directions ($y\rightarrow x$) or for close tangency points ($y\rightarrow-x$) the distortion is close to $1$.
\end{remark}

\begin{open} Does the same holds for non-Radon planes which are smooth and strictly convex?
\end{open}

In \cite{Wu2} it is proved that a Minkowski plane is Euclidean if and only if the tangent segments drawn from any external point to the unit circle have equal length. Using Theorem \ref{distortion} we can give a simple proof of this in the case where the plane is smooth and strictly convex.

\begin{prop} The outer distortion functional is identically $1$ if and only if the norm is Euclidean.
\end{prop}
\begin{proof} 

First we argue that any plane whose outer distortion functional is identical to $1$ is Radon. Indeed, if this is the case, then in particular $x\dashv_B y$ implies that $\gamma(x,y) = 1$. According to Lemma \ref{propertiesouterfunctional} this characterizes Radon planes.\\

Now we can use Theorem \ref{distortion}. We have $\mathrm{cm}(x,x+y) = \mathrm{cm}(y,x+y)$ for all $(x,y) \in S\times S\setminus(D\cup D')$. Of course this equality also holds in $D\cup D'$. But this implies $[x,b(y)] = [y,b(x)]$ for any $x,y \in S$, and hence the function $\mathrm{cm}$ is symmetric. It follows that the norm is Euclidean.

\end{proof}

\section{Further topics}

\subsection{Almost Euclidean properties of cm}

Throughout this section we explore some geometric properties of the cosine function which are almost Euclidean, in some sense. 

\begin {prop} \label{equilateralprop} Assume that $(X,||\cdot||)$ is a Minkowski plane. If $x,y,z  \in X_o$ are the sides of an equilateral triangle oriented such that $x = y+z$, then we have 
\begin{align} \label{equilateraltriangle} \mathrm{cm}(x,y)+\mathrm{cm}(x,z)=1. \end{align}
\end {prop}

\begin{proof} We may assume, without loss of generality, that $x,y$ and $z$ are unit vectors. Thus,
\begin{align*} \mathrm{cm}(x,y)+\mathrm{cm}(x,z) = [y,b(x)] + [z,b(x)] = [x,b(x)] = 1,
\end{align*}
and this concludes the proof. 

\end{proof}

\begin{coro} If we define a cm-based symmetric cosine function by $\mathrm{ca}(x,y) =  \frac{\mathrm{cm}(x,y)+\mathrm{cm}(y,x)}{2}$, then in the hypothesis of the previous proposition we have 
\begin{align*}\mathrm{ca}(x,y)+\mathrm{ca}(x,z)+\mathrm{ca}(y,-z) = \frac{3}{2}.\end{align*}
\end{coro}

In \cite{bmt} it is proved that the sine function has Euclidean behavior within isosceles triangles: the angles of the base have the same sine. Although the cosine function does not have necessarily this property, it also has a Euclidean property for isosceles triangles: angles determined by an altitude drawn from the vertex which join equal sides have the same cosine. This is proved now.

\begin{prop} Let $\Delta\mathbf{oxy}$ be an isosceles triangle with $||x|| = ||y||$, and let $\mathrm{seg}[o,z]$ be an altitude in the Birkhoff sense ($z\dashv_By-x$). Thus, we have $\mathrm{cm}(z,x) = \mathrm{cm}(z,y)$.
\end{prop}

\begin{proof} Assume, without loss of generality, that $[z,y-x] > 0$. We just have to calculate 
\begin{align*} \mathrm{cm}(z,x) = \frac{[x,b(z)]}{||x||} = \frac{[x,y-x]}{||y-x||_a||x||} = \frac{[x,y]}{||y-x||_a||x||}, \ \mathrm{and}\\ \mathrm{cm}(z,y) = \frac{[y,b(z)]}{||y||} = \frac{[y,y-x]}{||y-x||_a||y||} = \frac{[x,y]}{||y-x||_a||x||}.
\end{align*}
And this finishes the proof.

\end{proof}

We finish this section presenting a characterization of the Glogovskii angular bisector in terms of the function $\mathrm{cm}$.

\begin{prop} Let $\left.[o,x\right>$ and $\left.[o,y\right>$ be two rays forming an angle, such that its Glogovskii bisector is $g$. Thus, $v \in g$ if and only if $\mathrm{cm}(\alpha x-v,v) = \mathrm{cm}(\beta y-v,v)$, where $\alpha,\beta \geq 0$ are the (unique) numbers for which $\alpha x - v \dashv_B x$ and $\beta y - v \dashv_B y$ respectively.
\end{prop}

\begin{proof} First, notice that $[\alpha x-v,x]$ and $[\beta y-v,y]$ have different signs, and we may assume that, without loss of generality, $[\beta y-v,v] > 0$. We have 
\begin{align*}\mathrm{cm}(\alpha x-v,v) = \frac{[v,b(\alpha x-v)]}{||v||} = \frac{[v,-x]}{||x||_a||v||} =  -\mathrm{sn}(v,x),
\end{align*}
and similarly we have $\mathrm{cm}(\beta y-v,v) = \mathrm{sn}(v,y)$. A little adaptation of Proposition 3.3 in \cite{bmt} shows that $v \in g$ if and only if $\mathrm{sn}(v,y) = -\mathrm{sn}(v,x)$. Thus, we have indeed the desired.

\end{proof}

For the Busemann angular bisector we have an almost Euclidean property. In the Euclidean plane, the length of a diagonal of a parallelogram with unit sides equals twice the cosine of the angle it determines with one of its adjacent sides. Clearly, if $\left.[o,z\right>$ is the Busemann angular bisector of the angle formed by the rays $\left.[o,x\right>$ and $\left.[o,y\right>$, where $x$ and $y$ are unit vectors, then $\mathrm{cm}(z,x) + \mathrm{cm}(z,y) = ||x+y||$.

\subsection{Symmetric cosine function}

Proposition \ref{symmetricprop} shows that the function $\mathrm{cm}$ is symmetric if and only if the plane is Euclidean. Within this subsection we construct a $\mathrm{cm}$-based symmetric cosine function and explore its properties. This approach was already studied by Shonoda and Weiss \cite{shonoda-weiss}. Mili\u{c}i\u{c} \cite{milicic2007b} also investigated a symmetric cosine function constructed from a certain non-symmetric cosine function other than ours.

\begin{definition} We define the function $\mathrm{cn}:X_{o}\times X_{o}\rightarrow\mathbb{R}$ in a smooth Radon plane to be
\begin{align*} \mathrm{cn}(x,y) =  \sqrt{ \mathrm{cm}(x,y) \cdot \mathrm{cm}(y,x)}.
\end{align*}
 \end{definition}

At this point the reader may be wondering why the function $\mathrm{cn}$ is defined only for Radon planes. The fact is that it is only well defined in such planes, as the next proposition shows. Notice that it is also a (until now missing) characterization of Radon planes in terms of $\mathrm{cm}$.

\begin{prop} Let $(X,||\cdot||)$ be a smooth normed plane. Then, the norm is Radon if and only if $\mathrm{cm}(x,y).\mathrm{cm}(y,x) \geq 0$ for any $x,y \in X_o$. 
\end{prop}

\begin{proof} We use Corollary 3 of \cite{martiniantinorms}, which states that a plane is Radon if and only if the following holds: for any $x,y \in X_o$, if $x \dashv_B \lambda x + y$ and $y\dashv_B \mu y + x$, then $\lambda\mu \geq 0$.\\

Let $x,y \in X_o$ be independent vectors, and write $b(x) = \lambda(\alpha x + y)$ and $b(y) = \sigma(\beta y + x)$ for non-zero $\lambda,\sigma \in \mathbb{R}$. Thus, we have $x \dashv_B \alpha x + y$ and $y \dashv_B \beta y + x$. Notice that
\begin{align*} 1 = [x,b(x)] = [x,\lambda(\alpha x + y)] = \lambda[x,y] \ \mathrm{and} \\ 1 = [y,b(y)] = [y,\sigma(\beta y + x)] = \sigma[y,x],
\end{align*}
and hence $\lambda\sigma \leq 0$. On the other hand,
\begin{align*} \mathrm{cm}(x,y)\cdot\mathrm{cm}(y,x) = \frac{[y,b(x)]}{||y||}.\frac{[x,b(y)]}{||x||} = -\frac{(\alpha\beta)(\lambda\sigma)[x,y]^2}{||x||.||y||},
\end{align*}
and this shows that $\alpha\beta$ has the same sign as $\mathrm{cm}(x,y)\cdot\mathrm{cm}(y,x)$. Clearly, this is also true when $x$ and $y$ are dependent, and hence we have the desired. 

\end{proof}

\begin{remark} If $(X,||\cdot||)$ is a Minkowski space whose dimension is greater than $2$, then the $\mathrm{cn}$ function is only well defined if $||\cdot||$ is Euclidean, and in this case it obviously equals the absolute value of the standard cosine. Indeed, it is known that if the norm induced by $||\cdot||$ in any plane of $X$ is Radon, then $||\cdot||$ is derived from an inner product (see \cite{blaschke}).  
\end{remark}

The following immediate corollary can be seen as a local characterization of smooth Radon planes among smooth normed planes, in the sense of the local characterization of the Euclidean plane among all normed planes given by Valentine and Wayment in \cite{valentine1971wilson}. As far as the authors know such a characterization was still missing in the literature. \\

\begin{coro} A smooth Minkowski plane $(X,||\cdot||)$ is Radon if and only if the function $g:X_o\times X_o\rightarrow\mathbb{R}$ given by
\begin{align*}g(x,y)= \mathrm{sgn}\left(\lim_{t\rightarrow 0}\frac{||x+ty||-||x||}{t}\right)\end{align*}
is symmetric.
\end{coro}

The next theorem shows that the function $\mathrm{cn}$, in some sense, expresses the Pythagorean trigonometric identity. This can be regarded as an ``almost Euclidean" property of Radon planes. Notice that this is related to the footnote at page 162 of \cite{Bus3}.

\begin{teo} In a Radon plane, for any $x,z \in X_o$ we have
\begin{align*}  \mathrm{cn}(x,z)^2 + \mathrm{cn}(x,b(z))^2= 1.
\end{align*}

\end{teo}

\begin{proof} Notice first that, since $\{z,b(z)\}$ is a conjugate base, we have $b^2(z) = -\frac{z}{||z||_a}$. Now, assuming (without loss of generality) that $x,z \in S$, we calculate
\begin{align*} \mathrm{cn}(x,z)^2 + \mathrm{cn}(x,b(z))^2 = [x,b(z)].[z,b(x)]+\frac{[b(z),b(x)]}{||b(z)||}.\frac{[x,-z]}{||z||_a}.
\end{align*}
Since $||z||_a||b(z)|| = |[b(z),z]| = 1$, we have
\begin{align*} \mathrm{cn}(x,z)^2 + \mathrm{cn}(x,b(z))^2 = \left[x[z,b(x)]+b(x)[x,z],b(z)\right],
\end{align*}
and clearly $x[z,b(x)]+b(x)[x,z] = z$. This concludes the proof.

\end{proof}

\begin{remark} What we actually proved is that for any smooth Minkowski space we have, whenever $\{z,b(z)\}$ is a conjugate base, $\mathrm{cm}(x,z).\mathrm{cm}(z,x) + \mathrm{cm}(x,b(z)).\mathrm{cm}(b(z),x) = 1$ for every $x \in X_o$.
\end{remark}

\subsection{Differentiation} \label{differentiation}

Thompson \cite[Section 8.4]{Tho} gave various differentiation formulas for trigonometric functions with respect to suitable parameters. This subsection is devoted to revisit, in some sense, his results, but now using the function $\mathrm{sn}$ (instead of $\mathrm{sm}$) and the map $b$. First, we consider the arc length parameter. The Minkowski arc length of a curve is defined in the usual approximation way using the norm $||\cdot||$. Throughout this section we assume that the unit circle $S$ has $C^2$ regularity, and we let $\gamma_{\partial B}:[0,l]\rightarrow S$ be a positively oriented (with respect to $[\cdot,\cdot]$) arc length parameterization (where $l$ denotes the arc length of the unit circle measured in the norm). 

\begin{prop}\label{propdiff} Assume that the functions $\mathrm{cm}$ and $\mathrm{sn}$ are restricted to $S\times S$. We denote by $\frac{d}{ds_1}$ and $\frac{d}{ds_2}$ the arc length derivatives with respect to the first and second entries, respectively. Then we have

\begin{align*} \frac{d}{ds_1}\mathrm{sn}(x,y) = \mathrm{sn}(b(x),y) \ \mathrm{and} \\ \frac{d}{ds_2}\mathrm{cm}(x,y) = \mathrm{cm}(x,b(y)).
\end{align*}

\end{prop}

\begin{proof} We just have to notice that $\frac{d\gamma_{\partial B}}{ds}(x) = \frac{b(x)}{||b(x)||}$. We calculate 
\begin{align*} \frac{d}{ds_1}\mathrm{sn}(x,y) = \frac{d}{ds}\left(\frac{[\gamma_{\partial B}(s),y]}{||y||_a}\right) = \frac{[b(x),y]}{||y||_a||b(x)||} = \mathrm{sn}(b(x),y) \ \mathrm{and}\\
\frac{d}{ds_2}\mathrm{cm}(x,y) = \frac{d}{ds}\left([\gamma_{\partial B}(s),b(x)] \right) = \frac{[b(y),b(x)]}{||b(y)||} = \mathrm{cm}(x,b(y)). 
\end{align*}

\end{proof}

It is also worth to mention that, following \cite{P-B} and \cite[Corollary 8.4.5]{Tho}, we have that the functions $\mathrm{cm}$ and $\mathrm{sn}$ in Radon planes yield solutions to certain second order ordinary differential equations. Indeed, let $\rho:[0,l]\rightarrow\mathbb{R}$ be given as

\begin{align*} \rho(s) = \left|\left|\frac{d}{ds}b\left(\gamma_{\partial B}(s)\right)\right|\right|.
\end{align*}

In a smooth Radon plane the map $s\mapsto b\left(\gamma_{\partial B}(s)\right)$ is a parameterization of the unit circle, but one may expect it fails to be an arc length parameterization if the plane is not Euclidean. (This is indeed true, as will be clear later.) Hence, the function $\rho$ denotes, in some sense, a type of ``distortion" of the unit circle. It describes a kind of simple harmonic motion equation satisfied by the sine and cosine functions of a smooth Radon plane. We have:

\begin{prop} In a smooth Radon plane the functions $s \mapsto \mathrm{sn}\left(\gamma_{\partial B}(s),x_0\right)$ and $s \mapsto \mathrm{cm}\left(x_0,\gamma_{\partial B}(s)\right)$, where $x_0 = \gamma_{\partial B}(0)$, are the solutions to the equation $f'' + \rho f = 0$ with respective initial conditions $f(0) = 0$ and $f'(0) = 1$; and $f(0) = 1$ and $f'(0) = 0$.
\end{prop}
\begin{proof} From Proposition \ref{propdiff} we have 
\begin{align*} \frac{d}{ds}\mathrm{sn}\left(\gamma_{\partial B}(s),x_0\right) = \mathrm{sn}\left(b\left(\gamma_{\partial B}(s)\right),x_0\right).
\end{align*}
Since we are working in a Radon plane, we have that $b\left(\gamma_{\partial B}(s)\right)$ is a parameterization of the unit circle. Thus, its derivative is a vector which points in the direction of $b\left(b\left(\gamma_{\partial B}(s)\right)\right)$, and hence in the direction (and orientation) of $-\gamma_{\partial B}(s)$. Therefore,\\
\begin{align*} \frac{d^2}{ds^2}\mathrm{sn}\left(\gamma_{\partial B}(s),x_0\right) = \frac{d}{ds}\mathrm{sn}\left(b\left(\gamma_{\partial B}(s)\right),x_0\right) = \frac{d}{ds}\left[b\left(\gamma_{\partial B}(s)\right),x_0\right] = -\rho(s).\left[\gamma_{\partial B}(s),x_0\right] = \\ = -\rho(s).\mathrm{sn}\left(\gamma_{\partial B}(s),x_0\right),
\end{align*}
and the argument is analogous to that of the other function.  Also, it is straightforward that the initial conditions are fulfilled by the studied functions.

\end{proof}

\begin{coro} Let $(X,||\cdot||)$ be a smooth Radon plane, and let $\rho:[0,l]\rightarrow\mathbb{R}$ be defined as above. Then $\rho$ is constant if and only if the norm is Euclidean.
\end{coro}

\begin{proof} The main theorem of the proof is that there is a unique solution for the ordinary differential equation $f'' + \rho f = 0$ with initial conditions $f(0) = 0$ and $f'(0) = 1$. If $\rho$ is constant, then the function $\mathrm{sn}$ is, up to re-parameterization, the standard Euclidean sine function. Therefore, our statement is an easy consequence of \cite[Theorem 4.1]{bmt}.

\end{proof}

\begin{remark} In some sense this means that non-Euclidean Radon planes have a ``second order distortion". Hence, these planes behave similarly to Riemannian manifolds, which locally admit an ``Euclidean up to first order" parameterization (see \cite[Chapter 2]{petersen}).
\end{remark}

Following Thompson's book, the other parameterization of the unit circle that we will consider is the one given by areas of circular sectors. But we will proceed a little differently. We simply consider the usual area measure given by the fixed determinant form $[\cdot,\cdot]$. If $y \in S$, then the area of the (oriented) circular sector between the rays $\left.\left[o,\gamma_{\partial B}(0)\right>\right.$ and $\left.[o,y\right>$ is calculated by 

\begin{align*} a\left(y,\gamma_{\partial B}(0)\right) = \frac{1}{2}\int_0^{s_y}[\gamma_{\partial B}(s),\gamma_{\partial B}'(s)]\ ds,
\end{align*}
where $s_y \in [0,l]$ is such that $y = \gamma_{\partial B}(s_y)$. Of course, setting $ a(s_y)=2a\left(y,\gamma_{\partial B}(0)\right)$ we get a monotone bijection from $[0,l]$ to $[0,2A]$, where $A$ is the area of the unit circle. Thus, we can consider the re-parameterization $\gamma_{A}:[0,2A]\rightarrow S$ given by $\gamma_A(a) = \gamma_{\partial B}(s(a))$.

\begin{prop} The area parameterization $\gamma_A$ is a parameterization by arc length in the antinorm.
\end{prop}
\begin{proof} Indeed, we have
\begin{align*} \frac{da}{ds}(s_y) = \frac{d}{ds}\int_0^{s_y}[\gamma_{\partial B}(s),\gamma_{\partial B}'(s)]\ ds = \left[y,\frac{b(y)}{||b(y)||}\right] = \frac{1}{||b(y)||}
\end{align*}
for any $y \in S$. Thus,
\begin{align*} \left.\frac{d\gamma_{\partial B}}{da}\right|_y = \left.\frac{d\gamma_{\partial B}}{ds}.\frac{ds}{da}\right|_y = \frac{b(y)}{||b(y)||}.||b(y)|| = b(y).
\end{align*}
Since $||b(y)||_a = 1$ for every $y \in S$ the statement follows.

\end{proof}

\begin{remark} This result can be interpreted as a differential version of the so-called ``Kepler Law" for (smooth) Minkowski planes. See \cite[Theorem 7]{martiniantinorms}. In a Radon plane the norm and antinorm arc length parameters and also the area parameter coincide.
\end{remark}

\noindent\textbf{Acknowledgements.} The first named author thanks to CAPES for partial financial support during the preparation of this manuscript. Both authors want to thank to Professors Gunter Weiss and Ralph Teixeira for valuable contributions during the preparation of this manuscript.

\bibliography{bibliography}

\end{document}